\documentclass[12pt,a4paper,psamsfonts]{amsart}
\usepackage{fancyhdr}
\usepackage{appendix}
\usepackage{amssymb,amscd,amsxtra,calc}
\usepackage{mathrsfs}
\usepackage{cmmib57}
\usepackage{multirow}
\usepackage[all]{xy}
\usepackage{longtable}
\usepackage[colorlinks,linkcolor=blue,anchorcolor=blue,citecolor=blue]{hyperref}
\usepackage{tikz}

\theoremstyle{plain}
    \newtheorem{thm}{Theorem}[section]

     \newtheorem{conjecture}[thm]{Conjecture}

    \newtheorem{lemma}[thm]{Lemma}
    \newtheorem{proposition}[thm]{Proposition}
    \newtheorem{question}[thm]{Question}
    \newtheorem{theorem}[thm]{Theorem}

\theoremstyle{definition}
    \newtheorem{definition}[thm]{Definition}

    \newtheorem{notation}[thm]{Notation}
    \newtheorem*{notation*}{Notation and Terminology}
      
    \newtheorem{remark}[thm]{Remark}
    
\theoremstyle{remark}

\newcommand{\arxiv}[1]{\href{https://arxiv.org/abs/#1}{{\tt arXiv:#1}}}

\newcommand{\bP}{\mathbb{P}}
\newcommand{\bQ}{\mathbb{Q}}
\newcommand{\bR}{\mathbb{R}}

\newcommand{\bZ}{\mathbb{Z}}
\newcommand{\bF}{\mathbb{F}}

\newcommand{\alb}{\operatorname{alb}}

\newcommand{\Aut}{\operatorname{Aut}}

\newcommand{\GL}{\operatorname{GL}}

\newcommand{\id}{\operatorname{id}}

\newcommand{\NS}{\operatorname{NS}}

\newcommand{\PEC}{\operatorname{PEC}}

\newcommand{\red}{\mathrm{red}}

\newcommand{\mstriangle}[1]{
\begin{tikzpicture}[x=0.3cm,y=0.3cm]
\draw (-0.4,-0.433) -- (1.4,-0.433);
\draw (-0.2,-0.7794) -- (0.7,0.7794);
\draw (1.2,-0.7794) -- (0.3,0.7794);
\end{tikzpicture}
}
\newcommand{\mssharp}[1]{
\begin{tikzpicture}[x=0.3cm,y=0.3cm]
\draw (-0.8,-0.5) -- (0.8,-0.5);
\draw (-0.8,0.5) -- (0.8,0.5);
\draw (-0.5,-0.8) -- (-0.5,0.8);
\draw (0.5,-0.8) -- (0.5,0.8);
\end{tikzpicture}
}

\makeatletter

\newcommand{\Rmnum}[1]{\expandafter\@slowromancap\romannumeral #1@}
\makeatother

\begin{document}

\title[Zariski dense orbit conjecture]{A note on Zariski dense orbit conjecture}
\author{Sichen Li}
\address{
School of Mathematics, East China University of Science and Technology, Shanghai 200237, P. R. China}
\email{\href{mailto:sichenli@ecust.edu.cn}{sichenli@ecust.edu.cn}}

\begin{abstract}
In this paper we first note a result of birational automorphisms with bounded degree of projective varieties related with the Zariski dense orbit conjecture (ZDO) and the Zariski density of periodic points.
Next, we give a reduced result of  ZDO for automorphisms of projective threefolds, and show ZDO for automorphisms of projective varieties $X$ with the irregularity $q(X)\ge\dim X-1$. 
\end{abstract}

\thanks{The research is supported by  Shanghai Sailing Program (No. 23YF1409300).}
\maketitle
\section{Introduction}
\subsection{Zariski dense orbit conjecture}
The following Zariski dense orbit  conjecture (ZDO) was proposed by Medvedev and Scanlon \cite[Conjecture 5.10]{MS14},  by Amerik, Bogomolov and Rovinsky \cite{ABR11} and strengthens a conjecture of S.-W. Zhang \cite{Zhang06}.
\begin{conjecture}
\label{MSC}
Let $X$ be a projective variety over an algebraically closed field $k$ of characteristic zero and $f: X\dasharrow X$ a dominant rational self-map.
Then either $k(X)^f\ne k$ or there is a point $x\in X_f(k)$ whose orbit $\mathcal O_f(x)$ is Zariski dense in $X(k)$.
\end{conjecture}
\begin{remark}
The conidtion $k(X)^f\ne k$ is the same as saying that there is nonconstant rational function $\psi: X\to Y:=\bP^1$ such that $\psi\circ f=\psi$.
So there is not any point $x\in X_f(k)$ such that $\mathcal O_f(x)$ is Zariski dense.
It is immediate to see that such a condtion is absolutely necessary in order to hope for conclusion in Conjecture \ref{MSC} to hold; the difficultly in Conjecture \ref{MSC} is to prove that such a condition is indeed sufficient for the existence of a Zariski dense orbit.
\end{remark}
\subsection{Historical note}
Amerik and Campana \cite{AC08} proved Conjecture \ref{MSC} when assume that $k$ is uncountable.
 In \cite[Corollary 9]{Amerik11}, Amerik proved the existence of non-preperiodic algebraic point when $f$ is of infinite order.
In fact, Conjecture \ref{MSC} is also true in positive characteristic, as long as the base field $k$ is uncountable (see \cite[Corollary 6.1]{BGR17}).
On the other hand, when the transcendence degree of $k$ over $\bF_p$ is smaller than the dimension of $X$, there are counterexamples to the corresponding variant of Conjecture \ref{MSC}  in characteristic $p$ (as shown in \cite[Example 6.2]{BGR17}).
In \cite{ABR11}, Amerik, Bogomolv and Rovinsky proved Conjecture \ref{MSC} under the assumption $k=\overline{\bQ}$ and $f$ has a fixed point $p$ which is smooth and such that the eigenvalues of $df|_p$ are nonzero and multiplicatively independent.

In \cite[Theorem 1.3]{BGT15}, Bell, Ghioca and Tucker proved that if $f$ is an automorphism without nonconstant invariant rational function, then there exists a subvariety of codimension 2 whose orbit under $f$ is Zariski dense.
 In \cite[Corollary 1.7]{BGRS17}, Bell, Ghioca, Reichstein and Satriano proved that the conjecture  for all smooth minimal 3-folds of Kodaria dimension 0 with sufficiently large  Picard number, contingent on certain conjectures in the minimal model program.

In \cite[Theorem 1.1]{Xie17}, Xie proved Conjecture \ref{MSC} for dominant polynomial endomorphism $f: \mathbb A^2\to\mathbb A^2$.
In \cite[Theorem 1.11]{Xie22},  Xie proved  Conjecture \ref{MSC} for $f=(f_1,\cdots, f_n): (\bP^1)^n\to (\bP^1)^n$, where the $f_i$'s are endomorphisms of $\bP^1$, in \cite[Theorem 1.11]{Xie22}.
See also \cite[Theorem 14.3.4.2]{BGT16}, where $f_i$'s are not post-critically finite; and Medvedev and Scanlon \cite[Theorem 7.16]{MS14} for endomorphism $f:\mathbb A^n\to\mathbb A^n$ with
\begin{equation*}
f(x_1,\cdots, x_n)=(f_1(x_1),\cdots, f_n(x_n)), f_i(x_i)\in k[x_i].
\end{equation*}

If $X$ is a (semi-) abelian variety and $f$ is a dominant self-map, Conjecture \ref{MSC} has been proved by Ghioca, Satriano and Scanlon in \cite[Theorem 1.2]{GS17} and \cite[Theorem 1.1]{GS19}.
In \cite[Theorem 1.4]{GX18}, Ghioca and Xie showed that if $(X,f)$ satisfies the strong ZD-property, then $(X\times\mathbb A_{k}^n,\psi)$ satisfies the strong ZD-property, where the dynamical system $\psi: X\times\mathbb A_k^n\dashrightarrow X\times\mathbb A_k^n$ is given by $(x,y)\longmapsto (\varphi(x), A(x)y)$ and $A\in\GL(n,k(X))$.
In general, Xie showed in \cite[Theorem 3.34]{Xie22} that if $(X,f)$ satisfies the AZD-property, then $(X,\mathbb A_k^n,\psi)$ satisfies the AZD-property.

When $X$ is an algebraic surface, and $f$ is a birational self-map, Conjecture \ref{MSC} has been proved by Xie in \cite[Corollary 3.33]{Xie22}.
On the other hand, Xie in \cite[Theorem 1.15]{Xie22} proved Conjecture \ref{MSC} when $f$ is a surjective endomorphism of a smooth projective surface.
And then Xie, Jia and Zhang \cite[Theorem 1.9]{JXZ23} proved Conjecture \ref{MSC} for a surjective endomorphism of a projective surface.

A dominant rational self-map on a projective variety is called $p$-cohomologically hyperbolic if the $p$-th dynamical degree is strictly larger than other dynamical degrees.
For such a map over $\overline{\bQ}$, Jia, Shibata, Xie and Zhang \cite[Theorem 1.12]{JSXZ21}, Matuszawa and Wang \cite{MW22} showed that there are $\overline{\bQ}$-points with Zariski dense orbits if $f$ is an 1-cohomologically hyperbolic map on a smooth projective variety.

Let $f: X\to X$ be a non-isomorphic surjective endomorphism of a smooth projective threefold $X$.
Meng and Zhang \cite{MZ23} proved that any birational minimal model program becomes $f$-equivariant after iteration, provided that $f$ is $\delta$-primitive.
Here, $\delta$-primitive means that there is no $f$-equivariant (after iteration) dominant rational map $\pi: X\dasharrow Y$ to a positive lower-dimensional projective variety $Y$ such that the first dynamical degree remains unchanged.
As an application, they reduced the Zariski dense orbit conjecture for $f$ to a terminal threefold with only $f$-equivariant Fano contractions.

Now we consider the variants of the Zariski dense orbit conjecture in positive characteristic proposed in  \cite[Conjecture 1.3]{GS21} and \cite[Section 1.6]{Xie23}.
Ghoica and Saleh \cite{GS21,GS22,GS23} proved  the conjecture dense orbit conjecture in positive characteristic for regular self-maps of the tori $\mathbb G_m^N$, the split semiabelian varieties, and the additive group scheme $\mathbb G_a^N$.
Now let $X$ be a projective variety over an algebraically closed field $\mathbf  k$ of positive characteristic over $\overline{\bF_p}$.
Xie proved \cite[Proposition 1.7]{Xie23} the Zariski dense orbit conjecture in positive characteristic when $\mathrm{trdeg}_{\overline{\bF_p}}\mathbf k\ge\dim X$.
When  $\mathrm{trdeg}_{\overline{\bF_p}}\mathbf k\ge1$,  Xie  proved in \cite[Theorem 1.10 and 1.11]{Xie23}  the Zariski dense orbit conjecture in positive characteristic  for an automorphism of a projective surface,  and an 1-cohomologically hyperbolic dominant endomorphism $f$ of a smooth projective variety of dimension $d\ge2$.
\subsection{Main results}
Let  $f:X\dashrightarrow X$ be a dominant rational map of a projective variety $X$ and $H$ an ample divisor on $X$ .
We define the {\it first dynamical degree $\delta_f$ of $f$} as
\begin{equation*}
           \delta_f=\lim_{n\to\infty}((f^n)^*(H)\cdot H^{\dim X-1})^{1/n},
\end{equation*}
where $H$ is a nef and big Cartier divisor of $X$ (cf. \cite{DS05}).

Let $\NS(X)$ be the Neron-Severi group of $X$ and $\NS_{\bR}(X)=\NS(X)\otimes_\bZ\bR$.
Recall that $f^*|_{\NS_{\bR}(X)}$ have three different type: elliptic, parabolic or hyperbolic.
These situations can be read on the growth of the iterates of $f^*$.
If $||\cdot||$ is any norm on $\NS_\bR(X)$, they correspond  respectively to the following situations: $||(f^*)||$ is bounded, $||(f^*)^n||\sim cn^k$ for some $k>0$ and $c>0$, and $||(f^*)^n||\sim \delta_f^n$ for $\delta_f>1$.

In \cite[Theorem 5.1]{Fakhruddin03}, Fakhruddin proved the Zariski density  of periodic points if $f$ is a polarized endomorphism of a projective variety.
 Xie \cite{Xie15} give a complete classification of birational surface maps whose periodic points are Zariski-dense.
 
In this paper we first note a result of Zariski dense orbits and Zariski density of periodic points for birational automorphisms with bounded degree of projective varieties as follows.
\begin{theorem}
\label{thm-bounded}
ZDO is true for birational automorphisms with bounded degree of projective varieties.	
Then ZDO is true for $f$.
Moreover, suppose that $f$ is an automorphism of a projective variety $X$ and one of the following cases holds:
\begin{enumerate}
	\item $X$ has the Picard number one.
	\item  $X$ is Fano.
	\item $\delta_f=1$ and the  pseduoeffective cone  is polyhedral.
\end{enumerate}
Then the following statements hold.
\begin{enumerate}
	\item ZDO is true for $f$.
	\item the set of periodic points is Zariski dense if and only if $f$ is of finite order.
\end{enumerate}
\end{theorem}
Under the framework of dynamics on projective varieties by Kawamata, Nakayama and Zhang \cite{Kawamata85,Nakayama10,NZ09,NZ10,Zhang16}, Hu and the author \cite{HL21},  in \cite{Li22} we  gave a reduced result of the Kawaguchi-Silverman conjecture \cite{KS16} for automorphisms of projective threefolds, and showed the Kawaguchi-Silverman conjecture for automorphisms of projective varieties $X$ with the irregularity $q(X)\ge\dim X-1$. 
Using the same argument in  \cite[Proof of Theorems 1.4 and 1.6]{Li22}, we give the same results for ZDO as follows.
\begin{theorem}\label{reduced-auto}
Let $f$ be an automorphism of a normal projective  threefold $X$ with only klt singularities.
Suppose $K_X\sim_\bQ 0$ or  $\kappa(X)=-\infty$.
Then we may reduce ZDO for $(X,f)$ to the following three cases:
\begin{itemize}
\item[(1)] $X$ is weak Calabi-Yau and $f$ is primitive;
\item[(2)] $X$ is a  rationally connected threefold;
\item[(3)] $X$ is a  uniruled threefold admitting a special MRC fibration over an elliptic curve.
\end{itemize}
\end{theorem}
\begin{proposition}
\label{q(X)>X-1}
Let $f$ be an automorphism of a normal projective variety $X$ with positive dimension.
Then ZDO is  true for $(X,f)$ if $q(X)\ge\dim X-1$.	
\end{proposition}
{\bf Acknowledgment.}
I would like to thank Yohsuke Matsuzawa for constant conversations to complete the original version of this article.
I would  like to thank  Sheng Meng and Junyi Xie for numerous discussions, Long Wang and Jason Bell for useful comments,  and Meng Chen and De-Qi Zhang for their encouragement.
Finally, I am thankful to the referee for a careful reading of the article and the many suggestions.
\section{Preliminaries}\label{Preliminaries}
\begin{notation}
Let $X$ be a projective variety over an algebraically closed field $k$.
Denote by $\PEC(X)$ the closure of the set of classes  of effective $\bR$-Cartier divisors in $\NS_{\bR}(X)$.
We say a cone of  $\PEC(X)$ is {\it polyhedral} if it is generated by finitely many vectors.

Let $f$ be a dominant rational self-map of $X$.
Denote by $k(X)^f$ the field of $f$-invariant rational functions on $X$.
Let $X_f(k)$ be the set of $x\in X(k)$ whose orbit $\mathcal O_f(x)=\{x, f(x),\cdots\}$  is well-defined.
 We say that $x$ is periodic if there exists $n\in\mathbb N$ such that $f^n(\alpha)=\alpha$.
 We say a dominant rational map $\pi: X\dasharrow Y$ is {\it $f$-equivariant} if there is a dominant rational self-map $g$ of $Y$ such that $\pi\circ f=g\circ \pi$.
We say $f$ is {\it imprimitive} \cite{Zhang09-JDG} if there exists a $f$-equivariant map $\pi: X\dasharrow Y$ with $\dim X>\dim Y>0$.
The map $f$ is called {\it primitive} if it is not imprimitive.

It is well-known that the automorphism group scheme $\Aut_X$ of a projective variety $X$ is locally of finite order over $k$ and $\Aut(X)=\Aut_X(k)$.
The reduced neutral component $(\Aut_X^0)_{\red}$ of $\Aut_X$ is a smooth algebraic group over $k$ (cf \cite[\S7]{Brion17}).
Denote $(\Aut_X^0)_{\red}(k)$ by $\Aut_0(X)$.
\end{notation}
Now we recall  the definitions of ZD-property and strong ZD-property in \cite{Xie22}.
\begin{definition}
\cite[Definition 1.5]{Xie22}
We say that a  pair $(X,f)$ satisfies the \emph{ZD-property}, if either $k(X)^f\ne k$ or there exist a point $p\in X(k)$  whose orbit $\mathcal O_f(p)$ under $f$ is well defined and Zariski dense in $X$.
\end{definition}
\begin{definition}
\cite[Definition 1.6]{Xie22}
We say that a pair $(X,f)$ satisfies the \emph{strong ZD-property}, if either $k(X)^f\ne k$ or for every Zariski dense open subset $U$ of $X$, there exist a point $p\in X(k)$ whose orbit $\mathcal O_f(p)$ under $f$ is well-defined, contained in $U$ and Zariski dense in $X$.
\end{definition}
\begin{remark}
It is obvious that the strong ZD-property implies the ZD-property.	
\end{remark}
\begin{proposition}
\label{birational}
(cf. \cite[Proposition 2.2]{Xie22})
Let $\pi: X\dasharrow Y$ be a generically finite dominant rational map.
Let $f: X\to  X$ and $g: Y\to Y$ be surjective endomorphisms such that $g\circ\pi=\pi\circ f$.
Then $(X,f)$ satisfies the Strong ZD-property if and only if $(Y,g)$ satisfies the Strong ZD-property. 
\end{proposition}
\begin{proof}
By taking the graph of $\pi$, it suffices for us to consider the case when $\pi$ is a generically finite surjective morphism since the strong ZD-property is invariant under birational conjugation by \cite[Proposition 2.2]{Xie22}.
By \cite[Lemma 2.1]{Xie22}, $k(X)^f\ne k$ if and only if $k(Y)^g\ne k$.
Now take a point $x\in X(k)$.
Then $\mathcal O_f(x)$ is Zariski dense in $X(k)$ if and only if $\mathcal O_g(\pi(x))$ is Zariski dense in $Y(k)$.
\end{proof}
A normal projective variety $X$ is said to be {\it $Q$-abelian} if there is a finite surjective morphism $\pi: A\to X$ \'etale in codimension 1 with $A$ being an abelian variety.
\begin{theorem}
\label{Q-abelian}
ZDO is true for  surjective endomorphisms $f$ of a  $Q$-abelian variety $X$.
\end{theorem}
\begin{proof}
There exists a finite surjective morphism $\pi: A\to X$ with $A$ being an abelian variety, such that $f$ lifts to a surjective endomorphism $f_A: A\to A$ by \cite[Corollary 8.2]{CMZ20}.
Note that $(A,f_A)$  satisfies the strong ZD-property by \cite[Theorem 1.14 and Corollary 3.31]{Xie22}.
Therefore, ZDO is true for $(X, f)$ by Proposition \ref{birational}.
\end{proof}
\section{Proof of Theorem \ref{thm-bounded}}
\begin{theorem}
\cite[Theorem 2]{Rosenlicht56}
\label{thm-Rosenlicht}
Consider the action of an algebraic group $G$ on an irreducible algebraic variety $X$ defined over an algebraically closed field $k$ of characteristic 0.
Then exists a $G$-invariant dense open subvariety $X_0\subseteq X$ and a $G$-equivariant morphism $g: X_0 \to Z$ (where $G$ acts trivially on $Z$), with the following properties:
\begin{enumerate}
	\item for each $x\in X_0(k)$, the orbit $G\cdot x$ equals the fiber $g^{-1}(g(x))$, and 
	\item $g^*k(Z)=k(X)^G:=\{\psi\in k(X): \psi\circ h=\psi \text{ for each } h\in G \}$.
\end{enumerate}
In particular, if there is non nonconstant fibration fixed by $G$, then for each $x\in X_0(k)$, we have $G\cdot x=X_0$ is Zariski dense in $X$.
\end{theorem}
The following is a variant of Weil's regularization theorem (cf. \cite[Theorem 2.6]{Cantat14}).
\begin{theorem}
\label{thm-Weil-regular}
Let $f$ be a birational  self-map of a projective variety $X$.
If $f$ is of bounded degree, there exist a smooth projective variety $\widetilde{X}$ and a  birational map $\tau: \widetilde{X}\dasharrow X$ such that $\widetilde{f}=\tau^{-1}\circ f\circ \tau\in\Aut(\widetilde{X})$ and $\widetilde{f}^s\in\Aut_0(\widetilde{X})$ with some $s\in\bZ_{>0}$.
\end{theorem}
\begin{proposition}
(cf. \cite[Proposition 6.2]{Xie15})
\label{prop-zdp}
Let $f$ be an automorphism of a projective variety $X$.
If $f^*D\equiv D$ for some big $\bR$-divisor $D$, then $f^s\in\Aut_0(X)$, and the set of periodic points is Zariski dense if and only if $f$ is of finite order.
\end{proposition}
\begin{proof}
Let $f^*D\equiv D$ for some big $\bR$-divisor $D$.
Then
$$
      f^s\in\Aut_{[D]}(X):=\{f\in \Aut(X)| f^*[D]=[D]\}.
$$
So there is an integer $n>0$ such that $f^n\in\Aut_0(X)$ by a Fujiki-Lieberman type theorem (cf. \cite{Fujiki78,Lieberman75}, see also \cite[Theorem 1.4]{Li20}).
So $f^n|_{\NS_{\bR(X)}}=\id$.
Then it follows from \cite[Proposition 6.2]{Xie15}.
\end{proof}
\begin{proof}[Proof of Theorem \ref{thm-bounded}]
By Theorem \ref{thm-Weil-regular} and Proposition \ref{birational}, we may assume that $f$ is an automorphism of a smooth projective variety $X$ and $f^s\in\Aut_0(X)$ with  some $s\in\bZ_{>0}$.
After replacing $f$ by a suitable iterate, we assume that $f\in\Aut_0(X)$ and  it is of infinite order.
Let $G$ be the Zariski closure of the cyclic group spanned by $f$ inside in $\Aut_0(X)$.
Then by Theorem \ref{thm-Rosenlicht}, if $f$ does not fix a nonconstant fibration, then there is $x\in X(k)$ such that $G\cdot x$ is dense in $X$, and therefore $\mathcal O_f(x)$ is Zariski dense in $X$ as well.

Now we assume that $f$ is an automorphism.
By Proposition \ref{prop-zdp}, if $f^*[D]=[D]$ for some big $\bR$-divisor $D$, then $f^s\in\Aut_0(X)$ (and so ZDO is true for $f$), and the set of periodic points is Zariski dense if and only if $f$ is of finite order.
Therefore, it suffices to show  that $f$ fixes a big divisor.

(1) Assume that $\rho(X)=1$.
It is well-known that $f^*D\equiv\delta_fD$ for some nef $\bR$-divisor $D$.
Then $D$ is big as $\rho(X)=1$.
Therefore,  $(f^*D)^{\dim X}=(\deg f)D^n=(\delta_f)^nD^n$.
This yields that $\deg f=(\delta_f)^{\dim X}$ as $D^n>0$.
Then $\delta_f=1$.
So $f$ fixes a big divisor.

(2) Assume that $X$ is Fano.
Note that $-K_X$ is ample and  $f^*(-K_X)\equiv(-K_X)$.
So $f$ fixes a big divisor.

(3) Assume that  $\PEC(X)$ is polyhedral.	
By \cite[Theorem 1.2(1)]{Li20}, we may assume that  all roots  of  the characteristic polynomial  of  $f^*|_{\NS_\bR(X)}$  are 1 after replacing $f$ by $f^s$ with $s\in\bZ_{>0}$.
Therefore, $f^*|_{\PEC(X)}=\id$ as $\PEC(X)$ is a polyhedral subcone of $\NS_{\bR}(X)$.
So $f^*$ fixes a big class $[D]$, i.e. $f^*[D]=[D]$.
\end{proof}
\section{Proof of Theorem \ref{reduced-auto} and Proposition \ref{q(X)>X-1}}
\label{projective 3-folds}
\begin{proposition}
\label{X=Y+1}
\cite[Lemma 5.1]{BGR17}
Let $\pi: X\to Y$ be a surjective morphism of projective varieties.
Let $f: X\to   X$ and $g: Y\to Y$ be surjective endomorphisms such that $g\circ\pi=\pi\circ f$.
Suppose $\dim X=\dim Y+1$, $f$ is an automorphism and $(Y,g)$ satisfies the  ZD-property.
Then $(X,f)$ satisfies the ZD-property.
\end{proposition}
\begin{proposition}
\label{threefold}
Let $\pi: X\to Y$ be a surjective morphism of normal projective varieties with $3=\dim X>\dim Y\ge1$.
Let $f: X\to  X$ and $g: Y\to  Y$ be  surjective endomorphisms such that $g\circ\pi=\pi\circ f$.
Suppose $f$ is an automorphism.
Then to show ZDO holds for $(X,f)$, we only to assume that $Y$ is $\bP^1$ or an elliptic curve.	
In particular, if $\kappa(X)\ge0$, then $Y$ is an elliptic curve.
\end{proposition}
\begin{proof}
When $\dim Y=2$, then ZDO is true for $(Y,g)$ by \cite[Theorem 1.9]{JXZ23}.
Then ZDO is true for $(X,f)$ by Proposition \ref{X=Y+1}.
Then we may assume that $\dim Y=1$.
It is well-known that $Y$ does  not have any dense $g$-orbit if $g(Y)\ge2$, i.e. $\kappa(Y)>0$.
Therefore, $Y$ is $\bP^1$ or an elliptic curve.
If $\kappa(X)\ge0$, then we end  the proof of Proposition \ref{threefold} by \cite[Theorem 3.4]{Li22}.
\end{proof}
\begin{proof}[Proof of Proposition \ref{q(X)>X-1}]
By \cite[Proposition 5.1]{CLO22} (or \cite[Proposition 3.7]{LM21}), we may assume that the Albanese map $\alb: X \to E$ is surjective.
By taking the graph of $\alb$, it suffices to consider the case $\alb$ is morphism by Proposition \ref{birational}.
 Notice that $f$ descents to a surjective endomorphism $f_A$ of $A$ by the universal of the Albanese morphism.
If $\dim A=\dim X$, then Conjecture \ref{MSC} is true for $(X,f)$ by Proposition \ref{birational} and Theorem \ref{Q-abelian}.
If $\dim A=\dim X-1$, then  Conjecture \ref{MSC} is true for $(X,f)$ by Proposition \ref{X=Y+1}  and Theorem \ref{Q-abelian}.
\end{proof}
We recall the definition of a weak Calabi-Yau variety in \cite[Definition 2.4]{HL21}.
\begin{definition}\label{defn-wcy}
A normal projective variety $X$ is called a \emph{weak Calabi-Yau variety}, if
\begin{itemize}
\item $X$ has only canoncial singularities;
\item the canonical divisor $K_X\sim 0$; and
\item the \emph{augemented irregularity} $\widetilde{q}(X)=0$.
\end{itemize}
Here, the augumented irregularity $\widetilde{q}(X)$ of $X$ is defined as the supremum of $q(Y)$ of all normal projective varieties $Y$ with finite surjective morphism $Y\to X$, \'etale in codimension one. 
Namely,
$$\widetilde{q}(X)=\sup\big\{ q(Y) : Y\to X \text{ is  a finite surjective and \'etale in codimension one}\big\}.$$
\end{definition}
\begin{proposition}
\cite[Proposition 4.14]{Nakayama10}
\label{Chow-red}
Let $\pi:X\dasharrow Y$ be a dominant rational map from a projective variety $X$ to a normal projective variety $Y$.
Then there exists a normal projective variety  $T$ and a birational map $\mu: Y\dasharrow T$ satisfying the following conditions:
\begin{itemize}
\item[(1)] The graph $\gamma_{\mu\circ\pi}:\Gamma_{\mu\circ\pi}\to T$ of $\mu\circ \pi$ is equi-dimensional.
\item[(2)] Let $\mu': Y\dasharrow T'$ be a birational map to another normal projective variety $T'$ such that the graph $\gamma_{\mu'\circ\pi}:\Gamma_{\mu'\circ\pi}\to T'$ of $\mu'\circ\pi$ is equi-dimensional.
Then there exists a birational morphism $\nu: T'\to T$ such that $\mu=\nu\circ \mu'$.
\end{itemize}
\end{proposition}
We call the composition $\mu\circ\pi: X\dasharrow T$ above satisfying Proposition \ref{Chow-red} (1)-(2) the \emph{Chow reduction} of $\pi: X\dasharrow Y$, which is unique up to isomorphism.
\begin{lemma}\label{kappa=0}
Let $f$ be an automorphism of a normal projective  threefold $X$ with $K_X\sim_\bQ0$ and only klt singularities.
Then we may reduce ZDO for $(X,f)$ to the case that $f$ is a primitive automorphism of a weak Calabi-Yau threefold.
\end{lemma}
\begin{proof}
By \cite[Lemma 2.7]{HL21}, there exists a finite surjective morphism   $\pi: \widetilde{X} \to X$ and an automorphism $\widetilde{f}$ of $\widetilde{X}$ such that the following statements hold.
\begin{itemize}
	\item $X\cong Z\times A$ for a weak Calabi-Yau variety $Z$ and an abelian variety $A$.
	\item $\dim A=\widetilde{q}(X)$.
	\item There are automorphisms $\widetilde{f}_Z$ and $\widetilde{f}_A$ of $Z$ and $A$ respectively, such that the following diagram commutes:
$$
\xymatrix{
X  \ar[d]_f  &    \widetilde{X} \ar[l]_{\pi}  \ar[r]^{\cong}  \ar[d]_{\widetilde{f}} & Z\times A \ar[d]^{\widetilde{f}_Z\times \widetilde{f}_A}\\
X   &  \widetilde{X} \ar[r]^{\cong}  \ar[l]_{\pi} & Z\times A.\\
}
$$
\end{itemize}
If $\widetilde{q}(X)=3$, then it follows from Theorem \ref{Q-abelian} and Proposition \ref{birational}.
If $\dim Z>0$, then $\dim Z\ge2$ since  $\widetilde{q}(Z)=0$ and $K_Z\sim0$.
If $\dim Z=2$, then $\dim A=1$.
Consider the natural projection $\mathrm{pr}_1: \widetilde{X}\to Z$.
Then Propositions   \ref{birational} and \ref{X=Y+1} implies  ZDO for $(X,f)$.
Now assume that $\dim Z=3$ and  $f$ is imprimitive.
Then there is a rational map $\pi: X\dasharrow Y$ and a birational map $g: Y\dasharrow Y$ such that $\pi\circ f=g\circ \pi$.
By Proposition \ref{Chow-red}, there exists a birational morphism $\mu: Y\to Z$ such that $\pi'=\mu\circ \pi: X\dasharrow Z$ is the Chow reduction of $\pi$.
Then $f$ descents to  an automorphism $h$ of $Z$ by Proposition \ref{Chow-red}.
 By taking the graph of $\pi'$, it suffices to consider the case when $\pi'$ is morphism by Proposition  \ref{birational}.
By Proposition \ref{threefold}, we may assume that  $Y$ is  an elliptic curve.
This completes the proof  of  Lemma \ref{kappa=0} as $X$ has trivial Albanese.
\end{proof}
The following {\it special MRC fibration} is due to Nakayama \cite{Nakayama10}.
\begin{definition}
\cite[Definition 2.10]{HL21}
Given	a projective variety $X$, a dominant rational map $\pi: X\dashrightarrow Z$ is called the  special MRC fibration of $X$, if it satisfies the following conditions:
\begin{itemize}
\item[(1)] The graph $\Gamma_\pi\subseteq X\times Z$ of $\pi$ is equidimensional over $Z$.
\item[(2)] The general fibers of $\Gamma_\pi\to Z$ are rationally connected.
\item[(3)] $Z$ is a non-uniruled normal projective variety.
\item[(4)] If $\pi':X\dasharrow Z'$ is a dominant rational map satisfying (1)-(3), then there is a birational morphism $v: Z'\to Z$ such that $\pi=v\circ\pi'$.
\end{itemize}
\end{definition}
The existence and the uniqueness (up to isomorphism) of the special MRC fibration is proved in \cite[Theorem 4.18]{Nakayama10}.
One of the crucial advantages of the special MRC is the following descent property.
\begin{theorem}\label{descent}
\cite[Lemma 2.11]{HL21}
Let $\pi: X\dasharrow Z$ be the special MRC fibration, and $f\in\Aut(X)$.
Then there exists a  birational morphism $p: W\to X$ and an automorphism $f_W\in\Aut(W)$ and an equidimensional surjective morphism $q: W\to Z$ satisfying the following conditions:
\begin{itemize}
\item[(1)]  $W$ is a normal projective variety.
\item[(2)] A general fiber of $q$ is rationally connected.
\item[(3)] $W$ admits $f_W$-equivariant fibration over $X$ and $Z$.	
\end{itemize}
\end{theorem}
\begin{lemma}\label{kappa<0}
 Let $f$ be an automorphism of a  normal projective  threefold $X$ with only klt singularities and $\kappa(X)=-\infty$.
Then we may reduce  ZDO for $(X,f)$ to the following two cases:\begin{itemize}
\item[(1)]  $X$ is a  rationally connected threefold;
\item[(2)] $X$ is a  uniruled threefold admitting a special MRC fibration over an elliptic curve.
\end{itemize}
\end{lemma}
\begin{proof}
Consider the special MRC fibration $\pi: X\dasharrow Z$.
Then by Theorem \ref{descent} and Proposition \ref{birational}, we  assume that $\pi: X\to Z$ is an $f$-equivariant morphism and the general fibrer of $q$ is rationally connected.
If $\dim Y=0$, then $X$ is rationally connected.
Now assume that $\dim Y>0$.
Then the proof follows from Proposition \ref{threefold} as $Y$ is non-uniruled.
\end{proof}
\begin{proof}[Proof of Theorem \ref{reduced-auto}]
The proof follows from Lemmas \ref{kappa=0} and \ref{kappa<0}.
\end{proof}
Now let $f$ be an automorphism of a normal projective variety $X$ with $K_X\sim0$ and $q(X)>0$.
By \cite[Lemma 2.7]{HL21}, there is a quasi-\'etale cover $\pi:\widetilde{X}\to X$ and a lifting $\widetilde{f}\in \Aut(\widetilde{X})$ of $\widetilde{X}$ such that $\widetilde{X}=A\times Z$, $A$ and $Z$ respectively are an abelian variety and a weak Calabi-Yau variety, and $\widetilde{f}=f_A\times f_Z$ with $f_A\in\Aut(A)$ and $f_Z\in\Aut(Z)$.
Here, $\dim Z\ge2$.
Then an interesting question is asked as follows.
\begin{question}
\label{que-K0}
Let $X=A\times Z$ be the product of an abelian variety $A$ and a weak Calabi-Yau variety $Z$ with  $\dim Z\ge2$, and $f=f_A\times f_Z$ with $f_A\in\Aut(A)$ and $f_Z\in\Aut(Z)$.
Assume that AZD is true for $(Z,f_Z)$.
Is  AZD true for $(X,f)$?
\end{question}
\begin{remark}
When $\dim X=3$, Question \ref{que-K0} has an affirmative answer  by \cite[Theorem 1.9]{JXZ23} and Proposition \ref{threefold}.
\end{remark}

\end{document}